\documentclass[12pt]{amsart}

\usepackage{amssymb,enumerate,amsthm,yfonts,mathrsfs}
\usepackage{amsmath}
\usepackage{bbm}           
\usepackage{bm} 
\usepackage{eso-pic,graphicx}
\usepackage{tikz}

\usepackage[colorlinks=true, pdfstartview=FitV, linkcolor=blue, citecolor=blue, urlcolor=blue]{hyperref}

\makeatletter
\def\Ddots{\mathinner{\mkern1mu\raise\p@
\vbox{\kern7\p@\hbox{.}}\mkern2mu
\raise4\p@\hbox{.}\mkern2mu\raise7\p@\hbox{.}\mkern1mu}}
\makeatother

\def\Xint#1{\mathchoice
{\XXint\displaystyle\textstyle{#1}}%
{\XXint\textstyle\scriptstyle{#1}}%
{\XXint\scriptstyle\scriptscriptstyle{#1}}%
{\XXint\scriptscriptstyle\scriptscriptstyle{#1}}%
\!\int}
\def\XXint#1#2#3{{\setbox0=\hbox{$#1{#2#3}{\int}$}
\vcenter{\hbox{$#2#3$}}\kern-.5\wd0}}

\def\dashint{\Xint-}

\newtheorem{theorem}{Theorem}[section]

\newtheorem{conjecture}[theorem]{Conjecture}

\newtheorem{example}[theorem]{Example}

\theoremstyle{definition}
\newtheorem{remark}[theorem]{Remark}

\newtheorem{lemma}[theorem]{Lemma}

\def\al{{\alpha}}
\def\R{\mathbb R}

\def\Z{\mathbb Z}

\def\ra{\rightarrow}
\def\bey{\begin{eqnarray*}}
\def\eey{\end{eqnarray*}}
 
\numberwithin{equation}{section}
 \allowdisplaybreaks

\DeclareMathOperator{\supp}{supp}

\def\D{{\mathscr D}}

\def\M{{\mathcal M}}

\def\Sp{{\mathcal S}}

\begin{document}

\title[A fractional Muckenhoupt-Wheeden theorem]{A fractional Muckenhoupt-Wheeden theorem and its consequences}
\author{David Cruz-Uribe, SFO and Kabe Moen }

\thanks{The first author is supported by the Stewart-Dorwart faculty
  development fund at Trinity College and by
  grant MTM2009-08934 from the Spanish Ministry of Science and
  Innovation.  The second author is supported by NSF Grant 1201504}

 \address{David Cruz-Uribe, SFO, Department of Mathematics, Trinity College, Hartford, CT 06106, USA}

 \address{Kabe Moen, Department of Mathematics, University of Alabama, Tuscaloosa, AL 35487, USA}

\subjclass{42B25, 42B30, 42B35}

\keywords{Riesz potentials, fractional integral operators, Muckenhoupt
weights, sharp constants, two weight inequalities, bump conditions.}

\maketitle

\begin{abstract} In the 1970s Muckenhoupt and Wheeden made several
  conjectures relating two weight norm inequalities for the
  Hardy-Littlewood maximal operator to such inequalities for singular
  integrals.  Using techniques developed for the recent proof of the
  $A_2$ conjecture we prove a related pair of conjectures linking the
  Riesz potential and the fractional maximal operator.  As a
  consequence we are able to prove a number of sharp one and two
  weight norm inequalities for the Riesz potential.
\end{abstract}

\section{Introduction} 
In this paper we prove weight norm inequalities for the Riesz potential operator 
$$I_\al f(x)=\int_{\R^n}\frac{f(y)}{|x-y|^{n-\al}}\,dy, \qquad
0<\al<n.$$
Our main result is motivated by a pair of conjectures for singular
integrals due to Muckenhoupt and Wheeden, and to provide a foundation
for our work we first sketch these conjectures and the known results.

In the 1970s Muckenhoupt and Wheeden~\cite{muckenhouptP} conjectured that if $T$ is a
Calder\'on-Zygmund singular integral operator, then given a pair of
weights $(u,v)$, for $1<p<\infty$, 
\begin{equation} \label{eqn:sio1}
  T : L^p(v) \rightarrow L^p(u)
\end{equation}
provided that the Hardy-Littlewood maximal operator satisfies
\begin{gather}
M : L^p(v) \rightarrow L^p(u) \label{eqn:max1} \\
M : L^{p'}(u^{1-p'})\rightarrow L^{p'}(v^{1-p'}) \label{eqn:max2}. 
\end{gather}
Further, they conjectured that 
\begin{equation} \label{eqn:sio2}
  T : L^p(v) \rightarrow L^{p,\infty}(u)
\end{equation}
provided that \eqref{eqn:max2} holds.    Originally, they made these conjectures
for the Hilbert transform, but they were soon extended to
Calder\'on-Zygmund singular integrals.     While extremely suggestive
and true in many important cases,
both of these conjectures are false.  A counter-example to  the strong
type conjecture was found by Reguera and
Scurry~\cite{reguera-scurryP}, and this was extended to the weak-type
conjecture by the first author, Reznikov and Volberg~\cite{CRV2012}.  

However, a version of these conjectures is true in the
off-diagonal case.  If $1<p<q<\infty$, then the first author, Martell
and P\'erez~\cite{cruz-martell-perezP} showed that 
\begin{equation} \label{eqn:sio3}
  T : L^p(v) \rightarrow L^q(u)
\end{equation}
provided that 
\begin{gather}
M : L^p(v) \rightarrow L^q(u) \label{eqn:max3} \\
M : L^{q'}(u^{1-q'})\rightarrow L^{q'}(v^{1-q'}) \label{eqn:max4}. 
\end{gather}
and that 
\begin{equation} \label{eqn:sio4}
  T : L^p(v) \rightarrow L^{q,\infty}(u)
\end{equation}
provided that \eqref{eqn:max4} holds.   In fact, they proved a
quantitative version of this result in a slightly different form.  Let
$\sigma=v^{1-q'}$; then they showed that 
\begin{gather}
\|T(\,\cdot\,\sigma)\|_{L^p(\sigma)\ra L^{q,\infty}(u)}\lesssim
\|M(\,\cdot\, u)\|_{L^{q'}(u)\ra L^{p'}(\sigma)} \notag \\
\intertext{and}
\|T(\,\cdot\,\sigma)\|_{L^p(\sigma)\ra
  L^{q}(u)}\lesssim\|M(\,\cdot\,\sigma)\|_{L^p(\sigma)\ra L^{q}(u)}
+ \|M(\,\cdot\, u)\|_{L^{q'}(u)\ra L^{p'}(\sigma)}. \label{eqn:sio-max-strong}
\end{gather}
Replacing $f$ by $f/\sigma$ or $f/u$ yields inequalities in the form
given above.  This formulation has two advantages.  First, the weights do
not change under duality.  More precisely,   if  $M$ were a linear, self-adjoint
operator, then the inequality gotten from~\eqref{eqn:max3} by duality
would be~\eqref{eqn:max4}.  However, in the new formulation, the two
norm inequalities on the right-hand side of
~\eqref{eqn:sio-max-strong} would be dual.     Even though
the maximal operator is not linear, we will abuse terminology and
continue to refer to these as dual inequalities.  Second, this
formulation makes it easier to consider weights $v$ that are equal to
infinity on a set of positive measure, replacing it with a weight that
is zero.   Hereafter we will formulate all of our weighted norm
inequalities in this way.  

\bigskip

Our main result is an extension of these off-diagonal results to the
case of Riesz potentials with the Hardy-Littlewood maximal operator
replaced by the fractional maximal operator of Muckenhoupt and
Wheeden~\cite{muckenhoupt-wheeden74}:  
$$M_\al f(x)=\sup_{Q\ni x} |Q|^{\frac{\al}{n}}\,\dashint_Q |f|\,dy,\qquad 0\leq \al<n.$$

\begin{theorem}\label{thm:main} Given $0<\al<n$, $1<p<q<\infty$, and a
  pair of weights $(u,\sigma)$, then 
$$\|I_\al(\,\cdot\,\sigma)\|_{L^p(\sigma)\ra L^{q,\infty}(u)}\simeq \|M_\al(\,\cdot\, u)\|_{L^{q'}(u)\ra L^{p'}(\sigma)}$$
and 
$$\|I_\al(\,\cdot\,\sigma)\|_{L^p(\sigma)\ra
  L^{q}(u)}\simeq\|M_\al(\,\cdot\,\sigma)\|_{L^p(\sigma)\ra L^{q}(u)}+
\|M_\al(\,\cdot\, u)\|_{L^{q'}(u)\ra L^{p'}(\sigma)}.$$
In both inequalities the constants depend on $n$, $\alpha$, $p$ and $q$.
\end{theorem}

An open question is whether Theorem~\ref{thm:main} is true in the case
$p=q$.  Given the parallels between Riesz potentials and singular
integrals this seems doubtful and so we frame the conjecture in the
negative.   

\begin{conjecture} \label{conj:forlorn}
Theorem~\textup{\ref{thm:main}}  is false when $p=q$:  there exists a
pair $(u,\sigma)$ such that $\|M_\al(\,\cdot\, u)\|_{L^{p'}(u)\ra
  L^{p'}(\sigma)}<\infty$ but
$\|I_\al(\,\cdot\,\sigma)\|_{L^p(\sigma)\ra L^{p,\infty}(u)}=\infty$.  
\end{conjecture}

\bigskip

The remainder of this paper is organized as follows.  In
Sections~\ref{section:applications} and~\ref{section:apps2}
we give applications of Theorem~\ref{thm:main} to sharp constant,
one weight norm inequalities and to two weight, $A_p$ bump
conditions.   We will also discuss some conjectures related
to~\ref{conj:forlorn} made by us in an earlier
paper~\cite{cruz-moen-2012}.  In Section~\ref{section:main-thm} we
prove Theorem~\ref{thm:main}.  Finally, 
in Sections~\ref{section:other-proofs} and~\ref{section:op2} we prove the
results from Sections~\ref{section:applications} and~\ref{section:apps2}.

Throughout this paper all notation is standard or will be defined as
needed.  By a cube we will always mean a cube whose sides are parallel
to the coordinate axes.  If we write $A \lesssim B$, then $A\leq cB$,
where the constant $c$ depends on $n$, $p$, $q$ and $\alpha$.  By $A
\simeq B$ we mean that $A \lesssim B$ and $B \lesssim A$.  

\section{Generalized one weight inequalities}
\label{section:applications}

Theorem~\ref{thm:main} shows that to prove strong and weak type norm
inequalities for the Riesz potential, we need to prove strong type
norm inequalities for the fractional maximal operator.  We will
consider two approaches.  In this section we give a generalization of
the sharp constant, one weight norm inequalities considered
in~\cite{cruz-moen-2012}.

Given $1<p<q<\infty$ and a pair of weights $(u,\sigma)$,  we define
\begin{multline*}A^\al_{p,q}(u,\sigma,Q)=|Q|^{\frac{\al}{n}+\frac1q-\frac1p}\left(\dashint_Q
  u\,dx\right)^{\frac1q}\left(\dashint_Q \sigma\,dx\right)^{\frac{1}{p'}} \\
  = |Q|^{\frac{\al}{n}+\frac1q-\frac1p}\|u^{\frac1q}\|_{q,Q}\|\sigma^{\frac{1}{p'}}\|_{p',Q}.\end{multline*}
Note that this functional is symmetric in $u$ and $\sigma$:
$$A^\al_{p,q}(u,\sigma,Q)=A^\al_{q',p'}(\sigma,u,Q).$$
It is well known (cf.~\cite[p.~115]{MR2797562}) that if 
$$[u,\sigma]_{A_{p,q}^\alpha}=\sup_Q
A^\al_{p,q}(u,\sigma,Q)<\infty,$$
then
$$M_\al(\,\cdot\,\sigma):L^{p}(\sigma)\ra L^{q,\infty}(u) \quad \text{and} \quad M_\al(\,\cdot\,u):L^{q'}(\sigma)\ra L^{p',\infty}(u).$$

The strong type inequality
$$M_\al(\,\cdot\,\sigma):L^p(\sigma)\ra L^q(u)$$
holds in this case 
if we assume also that the weight $\sigma$ satisfies a reverse H\"older
inequality; equivalently, if we assume that $\sigma$ is in the
Muckenhoupt class $A_\infty$.   This class can be defined in several ways.
Traditionally (see~\cite{garcia-cuerva-rubiodefrancia85}) we say that
$\sigma\in A_\infty$ if 
$$[\sigma]_{A^{\exp}_\infty}= \sup_QA^{\exp}_\infty(\sigma,Q)=\sup_Q\left(\dashint_Q \sigma\,dx\right) \,\exp\left(-\dashint_Q\log\sigma\,dx\right)<\infty.$$
This is now sometimes referred to as the exponential $A_\infty$
condition.  However, for the purposes of sharp constant estimates, an equivalent
definition is very useful:  $\sigma\in A_\infty$ if and only if
$$[\sigma]_{A^M_\infty}= \sup_Q A^M_\infty(\sigma,Q) = \sup_Q \frac{1}{\sigma(Q)}\int_Q
M(\sigma\chi_Q)(x)\,dx<\infty,$$
where $M$ is the Hardy-Littlewood maximal operator. 
This equivalent condition was discovered independently by
Fujii~\cite{MR0481968} and Wilson~\cite{MR883661,wilson89} (see also
\cite{wilson07}).   The importance of this condition is that
$[\sigma]_{A^M_\infty} \lesssim [\sigma]_{A^{\exp}_\infty}$ and in fact the constant $[\sigma]_{A^M_\infty}$ can be substantially smaller
\cite{beznosova-reznikovP,hytonen-perez-analPDE}.  

Using Theorem~\ref{thm:main} we give norm estimates for Riesz
potentials in terms of these quantities.  Our approach to this problem
is based on recent work on the sharp constants for singular integrals.
(For the history of these results,
see~\cite{hytonen-lacey-IUMJ,hytonen-perez-analPDE,lerner-IMRN2012}
and the references they contain.)  The natural approach when $p$ and
$q$ satisfy the Sobolev relationship $1/p-1/q=\alpha/n$, is to find
sharp estimates in terms of $[u,\sigma]_{A_{p,q}^\alpha}$.  This case
was studied in~\cite{MR2652182}.  Our goal here is to refine these
estimates and extend them to general $p<q$. Following the work of
Hyt\"onen and P\'erez~\cite{hytonen-perez-analPDE}, we find sharp
constants in terms of $[u,\sigma]_{A_{p,q}^\alpha}$,
$[u]_{A^M_\infty}$ and $[\sigma]_{A^M_\infty}$.  We also give an
alternative approach: following Lerner and the second
author~\cite{lerner2012,lerner-moenP}, we prove estimates in terms of a
mixed condition that combines the $A_{p,q}^\alpha$ and $A_\infty$
condition:
$$[u,\sigma]_{A_{p,q}^\alpha(u,\sigma)A^{\exp}_\infty(\sigma)^{\frac1q}} 
=\sup_Q A^\al_{p,q}(u,\sigma,Q)A^{\exp}_\infty(\sigma,Q)^{\frac1q}.$$

The next result gives both kinds of estimates for the fractional
maximal operator; We defer the proof until
Section~\ref{section:other-proofs}. 

\begin{theorem} \label{maximalbound} 
Given $0<\alpha<n$ and  $1<p\leq q<\infty$, suppose $(u,\sigma)\in
A_{p,q}^\alpha$ and $\sigma\in A_\infty$.  Then
\begin{equation} \label{eqn:maximalbound1}
\|M_\al(\,\cdot\,\sigma)\|_{L^p(\sigma)\ra L^q(u)}\lesssim 
[u,\sigma]_{A_{p,q}^\alpha(u,\sigma)A^{\exp}_\infty(\sigma)^{\frac1q}}
\end{equation}
and
\begin{equation} \label{eqn:maximalbound2}
\|M_\al(\,\cdot\,\sigma)\|_{L^p(\sigma)\ra L^q(u)}
\lesssim [u,\sigma]_{A_{p,q}}[\sigma]_{A^M_\infty}^{\frac1q}.
\end{equation}
\end{theorem}

Hereafter, we will refer to estimates like \eqref{eqn:maximalbound1}
as one supremum estimates, and estimates like \eqref{eqn:maximalbound2}
as two suprema estimates.  In general the one supremum estimates are incomparable with the two suprema estimates (see \cite{lerner2012,lerner-moenP} for examples).

As an immediate consequence of Theorems~\ref{thm:main}
and~\ref{maximalbound} we get the following estimates for Riesz
potentials.  

\begin{theorem}  \label{thm:onewt-weak}
Given $0<\alpha<n$ and $1<p< q<\infty$, suppose $(u,\sigma)\in
A_{p,q}^\alpha$ and $u\in A_\infty$.   Then
\begin{equation} \label{eqn:weakRP1}
\|I_\al(\,\cdot\,\sigma)\|_{L^p(\sigma)\ra L^{q,\infty}(u)}
\lesssim [\sigma,u]_{A_{q',p'}^\alpha(\sigma,u)A^{\exp}_\infty(u)^{\frac1{p'}}}
\end{equation}
and
\begin{equation} \label{eqn:weakRP2}
\|I_\al(\,\cdot\,\sigma)\|_{L^p(\sigma)\ra L^{q,\infty}(u)}
\lesssim [\sigma,u]_{A_{q',p'}^\alpha}[u]_{A^M_\infty}^{\frac{1}{p'}}.
\end{equation}
\end{theorem}

\begin{theorem} \label{thm:onewt-strong}
Given $0<\alpha<n$ and $1<p<q<\infty$, suppose $(u,\sigma)\in A_{p,q}^\alpha$
and $u,\,\sigma\in A_\infty$.  Then
\begin{equation} \label{eqn:strongRP1}
\|I_\al(\,\cdot\,\sigma)\|_{L^p(\sigma)\ra L^{q}(u)}
\lesssim [u,\sigma]_{A_{p,q}^\alpha(u,\sigma)A^{\exp}_\infty(\sigma)^{\frac1q}}
+[\sigma,u]_{A_{q',p'}^\alpha(\sigma,u)A^{\exp}_\infty(u)^{\frac{1}{p'}}}
\end{equation}
and
\begin{equation} \label{eqn:strongRP2}\|I_\al(\,\cdot\,\sigma)\|_{L^p(\sigma)\ra L^{q}(u)}
\lesssim [u,\sigma]_{A_{p,q}^\alpha}([u]_{A^M_\infty}^{\frac1{p'}}+[\sigma]_{A^M_\infty}^{\frac1q}).
\end{equation}
\end{theorem}

\begin{remark}
We proved inequalities~\eqref{eqn:weakRP2} and ~\eqref{eqn:strongRP2}
in~\cite{cruz-moen-2012} using a more complicated corona decomposition
argument.  Moreover, we only obtained results for $p$
and $q$ that satisfy the Sobolev relation.  
\end{remark}

Theorems~\ref{thm:onewt-weak} and~\ref{thm:onewt-strong} can be
thought of as generalizing one weight inequalities for the Riesz
potential.  The classical one weight norm inequalities for Riesz
potentials due to Muckenhoupt and Wheeden~\cite{muckenhoupt-wheeden74}
were for the case when $p$ and $q$ satisfy the Sobolev relation, and
there exists a weight $w$ such that $u=w^q$ and $\sigma=w^{-p'}$.  In
this case it follows from the $A_{p,q}^\alpha$ condition that both
$u$ and $\sigma$  are in $A_\infty$.   In this case we can
restate~\eqref{eqn:strongRP2} as 
\[ \|I_\alpha\|_{L^p(w^p)\ra L^q(w^q)} \lesssim
[w^q]_{A_{s(p)}}^{\frac1q}\big(
[w^q]_{A^M_\infty}^{\frac1{p'}}+[w^{-p'}]_{A^M_\infty}^{\frac1q} \big), \]
where $s(p) = 1+p/q'$ and we say that a weight $v$ is in the
Muckenhoupt class $A_p$ if 
\[  [v]_{A_p} = \sup_Q A_p(v,Q) = \sup_Q \dashint_Q v\,dx
\left(\dashint_Q v^{1-p'}\,dx\right)^{p-1} < \infty. \]
By interpolation we can give a result that is
in some sense an improvement of this inequality.  We again defer the proof to
Section~\ref{section:other-proofs}.

\begin{theorem} \label{thm:red-thm}
Given $0<\alpha<n$ and $1<p<n/\alpha$, define $q$ by
$1/p-1/q=\alpha/n$.  If $w^q\in A_r$ for some $r<s(p)=1+p/q'$, then
\begin{equation} \label{eqn:red-thm1}
\|I_\alpha\|_{L^p(w^p)\ra L^q(w^q)} \lesssim
[w^q]_{A_r}^{\frac1q} [w^q]_{A^M_\infty}^{\frac1{p'}}.
\end{equation}
\end{theorem}

The corresponding result for singular integrals was proved
in~\cite{lerner-moenP}.  A weaker version of
Theorem~\ref{thm:red-thm}, with the assumption  $w^q \in A_r$
replaced by the assumption that $w^q \in A_1$, was recently proved by
Recchi \cite{Recchi}.

In~\cite{lerner-moenP} it was conjectured that 
for singular integrals, the one supremum estimates corresponding
to~\eqref{eqn:weakRP1} and~\eqref{eqn:strongRP1} could be improved by
replacing $A^{\exp}_\infty(\sigma,Q)$ on the right-hand side 
with the smaller quantity $A^M_\infty(\sigma,Q)$.   They were able to prove a partial result involving
an additional log term.

We believe that the corresponding conjecture is true for Riesz
potentials.   We can prove a partial result in the classical one
weight case.   To state it we define the one supremum constant needed in this case, for a general weight:
\[ [w]_{(A_{p})^{\beta} (A^M_\infty)^{\gamma}} = \sup_Q  A_{p}(w,Q)^{\beta}
A^M_\infty(w,Q)^{\gamma}. \]

\begin{theorem} \label{thm:Mlog-Ainfty}
Given $0<\alpha<n$ and $1<p<n/\alpha$, define $q$ by
$1/p-~1/q~=~\alpha/n$.  If $w^q\in A_{s(p)}$, $s(p)=1+p/q'$ and $\sigma=w^{-p'}$, then
\begin{equation} \label{Meqn:log-Ainfty}
\|M_\alpha\|_{L^p(w^p) \ra L^{q}(w^q)}
\lesssim \Phi\big([w^{-p'}]_{A_{s(q')}}\big)^{\frac1q}
[w^{-p'}]_{(A_{s(q')})^{\frac{1}{p'}}(A^M_\infty)^{\frac1q}}, 
\end{equation}
where $\Phi(t)=1+\log(t)$.
\end{theorem}
The proof of Theorem \ref{thm:Mlog-Ainfty} requires a testing
condition for the fractional maximal function in \cite{MR2534183}.
Using this condition it is very similar to the argument
in~\cite{lerner-moenP}.  We sketch the details of the proof in
Section~\ref{section:other-proofs}.  Once again, as a consequence of
Theorem~\ref{thm:main} we have the following result for the Riesz
potential.

\begin{theorem} \label{thm:intlog-Ainfty}
Given $0<\alpha<n$ and $1<p<n/\alpha$, define $q$ by
$1/p-1/q=\alpha/n$.  If $w^q\in A_{s(p)}$, $s(p)=1+p/q'$, then
\begin{equation} \label{eqn:log-Ainfty1}
\|I_\alpha\|_{L^p(w^p) \ra L^{q,\infty}(w^q)}
\lesssim \Phi\big([w^q]_{A_{s(p)}}\big)^{\frac1{p'}} [w^q]_{(A_{s(p)})^{\frac1q}  (A^M_\infty)^{\frac1{p'}}}
\end{equation}
and
\begin{multline} \label{eqn:log-Ainfty2}
\|I_\alpha\|_{L^p(w^p) \ra L^{q}(w^q)}
\lesssim \Phi\big([w^q]_{A_{s(p)}}\big)^{\frac1{p'}}  [w^q]_{(A_{s(p)})^{\frac1q}  (A^M_\infty)^{\frac1{p'}}} \\ +
\Phi\big([w^{-p'}]_{A_{s(q')}}\big)^{\frac1q}[w^{-p'}]_{(A_{s(q')})^{\frac1{p'}}(A^M_\infty)^{\frac1q}},
\end{multline}
where $\Phi(t)=1+\log(t)$.
\end{theorem}

\section{Two weight inequalities via $A_p$ bump conditions}
\label{section:apps2}

If we do not assume that $u,\,\sigma \in A_\infty$, then the
$A_{p,q}^\alpha$ condition is no longer sufficient for the strong type
inequality for the fractional maximal operators or for the Riesz
potentials.    The construction is deferred until Section~\ref{section:op2}.

\begin{example} \label{example:bad-wt}
Given $0<\alpha<n$ and $1<p\leq q < \infty$, 
there exists a pair of weights $(u,\sigma)\in A^\alpha_{p,q}$ and a
function $f\in L^p(\sigma)$ such that $M_\alpha(f\sigma)\not\in L^q(u)$. 
\end{example}

\begin{remark}
A similar example for the Hardy-Littlewood maximal operator (i.e.,
when $\alpha=0$) was constructed by Muckenhoupt and
Wheeden~\cite{muckenhoupt-wheeden76}.   While the existence of
Example~\ref{example:bad-wt}  is part of the folklore of harmonic analysis, to the
best of our knowledge one has never been published.  It is worth
noting that our example is considerably different from the one
constructed by Muckenhoupt and Wheeden.
\end{remark}

It is possible, however, to replace the $A^\alpha_{p,q}$ condition
with a stronger one defined using Orlicz norms.  This approach to
weighted norm inequalities is due to P\'erez~\cite{perez94,perez95}
and was motivated by the original Muckenhoupt-Wheeden conjectures.

To state these results we need to make some preliminary definitions.
(For further information, see~\cite[Section~5.2]{MR2797562}.)  A
Young function is a function $\Phi : [0,\infty) \rightarrow
[0,\infty)$ that is continuous, convex and strictly increasing,
$\Phi(0)~=~0$ and $\Phi(t)/t\rightarrow \infty$ as $t\rightarrow
\infty$. Define the localized Luxemburg average of $f$ over a cube $Q$ by
\[  \|f\|_{\Phi,Q} = 
\inf \left\{ \lambda > 0 : \dashint_Q \Phi\left(\frac{|f(x)|}{\lambda}\right)dx \leq 1 \right\}.  \]
When $\Phi(t)=t^p$, $1<p<\infty$, this becomes the $L^p$ norm and we write
$\|f\|_{\Phi,Q}=\|f\|_{p,Q}$.  
The associate function of $\Phi$ is the Young function
\[ \bar{\Phi}(t) = \sup_{s>0}\{ st - \Phi(s)\}. \]
Note that $\bar{\bar{\Phi}}=\Phi$.
A Young function $\Phi$ satisfies the $B_p$ condition if for some $c>0$,
\[ \int_c^\infty \frac{\Phi(t)}{t^p} \frac{dt}{t} < \infty. \]
Important examples of such functions are $\Phi(t)=t^{(rp')'}$, $r>1$, whose
associate function is $\bar{\Phi}(t)=t^{rp'}$, and
$\Phi(t)=t^{p}\log(e+t)^{-1-\epsilon}$, $\epsilon>0$, which have
associate functions $\bar{\Phi}(t)\simeq
t^{p'}\log(e+t)^{p'-1+\delta}$, $\delta>0$.  We refer to these
associate functions as power bumps and log bumps.  The $B_p$ condition
is important because it characterizes the $L^p$ boundedness of Orlicz
maximal operators, which in turn can be used to prove two weight
inequalities.  Define
\[ M_\Phi f(x) = \sup_{Q\ni x} \|f\|_{\Phi,Q}; \] 
then P\'erez \cite{MR1327936} showed that $M_\Phi$ is bounded on
$L^p(\R^n)$ if and only if $\Phi\in B_p$, and 
$$\|M_\Phi\|_{L^p\ra L^p}\lesssim 
 \left(\int^\infty_c\frac{\Phi(t)}{t^p}\frac{dt}{t}\right)^{1/p}.$$

For our results we need to generalize this to the fractional Orlicz
maximal operator.   Given $0<\alpha<n$ and a Young function $\Phi$, define
$$M_{\alpha,\Phi} f(x)=\sup_{Q\ni x}
|Q|^{\frac{\alpha}{n}}\|f\|_{\Phi,Q}.$$
We define the associated fractional $B_p$ condition as follows: given
$1<~p~<~n/\alpha$, let $1/q=1/p-\alpha/n$.  Then $\Phi \in
B^\alpha_p$ if
\[ \int_c^\infty \frac{\Phi(t)^{q/p}}{t^q}\frac{dt}{t}< \infty. \]
We prove the following result in Section~\ref{section:op2}.  

\begin{theorem} \label{thm:Mfracorlicz} Given $0< \alpha <n$ and 
  $1<p<{n}/{\alpha}$, define $1/q=1/p-\alpha/n$.
  Then for any  $\Phi\in B^\alpha_p$, $M_{\alpha,\Phi}:L^p(\R^n) \ra
  L^q(\R^n)$ and
\begin{equation}\label{eqn:fracorlicz}
\|M_{\alpha,\Phi}\|_{L^p\ra L^q} \lesssim
\left(\int^\infty_c\frac{\Phi(t)^{\frac{q}{p}}}{t^q}\frac{dt}{t}\right)^{\frac{1}{q}}.
\end{equation}
\end{theorem}

When $\alpha=0$ the two conditions coincide; if $\alpha>0$ then the
$B_p^\alpha$ condition is weaker. To see this, note that because the measure $\frac{dt}{t}$
on $(0,\infty)$ behaves in some sense like a counting measure, we have
$$\left(\int^\infty_c\frac{\Phi(t)^{\frac{q}{p}}}{t^q}\frac{dt}{t}\right)^{\frac{1}{q}}\lesssim  \left(\int^\infty_c\frac{\Phi(t)}{t^p}\frac{dt}{t}\right)^{1/p}.$$
Moreover, the Young function 
\[ \Phi(t)  = \frac{t^p}{\log(t)^{(1+\epsilon)\frac{p}{q}}} \]
is in $B_p^\alpha$ for any $\epsilon>0$ but is in $B_p$ only if
$\epsilon>q/p-1$.  Hence, $B_p\subsetneq B^\al_p$ if $\al>0$.

To state our results we introduce a new weight condition that 
is stronger than the $[u,\sigma]_{A_{p,q}^\alpha}$ condition,
replacing  the average on $\sigma$ by an Orlicz average:  we say that
$(u,\sigma)\in A_{p,q,\Phi}^\alpha$ if
$$[u,\sigma]_{A_{p,q,\Phi}^\alpha}=
\sup_Q
|Q|^{\frac{\al}{n}+\frac1q-\frac1p}\left(\dashint_Qu\,dx\right)^{\frac{1}{q}} 
\|\sigma^{\frac{1}{p'}}\|_{\Phi,Q}<\infty.$$
If we assume that $\Phi$ is such that $t^{p'}\leq C\Phi(ct)$, then $[u,\sigma]_{A_{p,q}^\alpha} \lesssim
[u,\sigma]_{{A_{p,q ,\Phi}^\alpha}}$.    This is always the case if
$\bar{\Phi}\in B_p$.    Note that this new condition lacks the
symmetry of the $A_{p,q}^\alpha$ condition since the Orlicz norm is always
applied to the second weight.  

This condition was introduced by P\'erez~\cite{perez94} (see
also~\cite[Section~5.6]{MR2797562}), who used it to prove strong type,
two weight
norm inequalities for the fractional maximal operator:
\begin{equation}\label{thm:perez-bump}
\|M_\al(\,\cdot\,\sigma)\|_{L^p(\sigma)\ra L^q(u)}
\lesssim [u,\sigma]_{A_{p,q,\Phi}^\alpha}\|M_{\bar\Phi}\|_{L^p\ra
  L^p}.
\end{equation}
When $p>q$ we improve his result both qualitatively and quantitatively, giving a larger class of Young functions and a sharper constant. 

\begin{theorem} \label{thm:Maltwoweight} 
Given $0\leq \al<n$ and
  $1<p\leq q<\infty$, define $\beta=n\big(\frac1p-\frac1q\big)$.  If
$\Bar{\Phi}\in B^\beta_p$ and the pair of weights $(u,\sigma)\in
  A_{p,q,\Phi}^\alpha$, then 
\begin{equation*}
\|M_\al(\,\cdot\,\sigma)\|_{L^p(\sigma)\ra L^q(u)}
\lesssim [u,\sigma]_{A_{p,q,\Phi}^\alpha}
\|M_{\beta,\bar\Phi}\|_{L^p\ra L^q}.
\end{equation*}
\end{theorem}

To see that this constant is sharper when $p>q$, we give two
examples.  If $\bar{\Phi}(t)=t^p\log(t)^{-(1+\epsilon)}$, $\epsilon>0$, then
a straightforward computation shows that 
the $B_p$ constant is approximately $\epsilon^{-1/p}$ but the
$B_p^\beta$ constant is approximately $\epsilon^{-1/q}$.  If
$\bar{\Phi}(t)=t^{(rp')'}$, $r>1$, then the $B_p$ constant is $(r')^{1/p}$
bu the $B_p^\beta$ constant is $(r')^{1/q}$.  (This
second example will be applied below.)

\medskip

As an immediate consequence of Theorems~\ref{thm:main}
and~\ref{thm:Maltwoweight} we get the corresponding two weight, weak and strong
type norm inequalities for Riesz potentials.

\begin{theorem} \label{weakbump} 
  Given $0< \al<n$ and $1<p< q<\infty$, let
  $\beta=n\big(\frac1p-\frac1q\big)$.  If  $\Bar{\Psi}\in B^\beta_{q'}$ and 
the pair $(u,\sigma) \in A_{q',p',\Psi}^\alpha$,  then
$$\|I_\al(\,\cdot\,\sigma)\|_{L^p(\sigma)\ra L^{q,\infty}(u)}
\lesssim [\sigma,u]_{A_{q',p',\Psi}^\alpha}\|M_{\beta,\bar\Psi}\|_{L^{q'}\ra L^{p'}}.$$
\end{theorem}

\begin{theorem}\label{strongbump} 
Given $0< \al<n$ and $1<p< q<\infty$, let
  $\beta=n\big(\frac1p-\frac1q\big)$.  If
$\Bar{\Phi}\in B^\beta_p$,  $\Bar{\Psi}\in B^\beta_{q'}$ and the pair  $(u,\sigma)$ satisfies 
$(u,\sigma)\in A_{p,q,\Phi}^\alpha$
and $(\sigma,u) \in A_{q',p',\Psi}^\alpha$,
then
\begin{equation} \label{eqn:strongbump1}
\|I_\al(\,\cdot\,\sigma)\|_{L^p(\sigma)\ra L^{q}(u)}\lesssim
[u,\sigma]_{A_{p,q,\Phi}^\alpha}\|M_{\beta,\bar\Phi}\|_{L^p\ra L^q}
+[\sigma,u]_{A_{q',p',\Psi}^\alpha}\|M_{\beta,\bar\Psi}\|_{L^{q'}\ra L^{p'}}.
\end{equation}
\end{theorem}

Theorem~\ref{strongbump} is referred to as a separated bump condition:
conditions of this kind were implicit in the work of P\'erez and were
introduced explicitly for singular integrals in~\cite{CRV2012} (see below).
This condition significantly improves the original, ``double bump'' result of P\'erez~\cite{perez94}, who showed
that 
\begin{equation}\label{eqn:dbump} 
\|I_\al(\,\cdot\,\sigma)\|_{L^p(\sigma)\ra L^{q}(u)} \lesssim
[u,\sigma]_{A^\al_{p,q,\Psi,\Phi}}
\|M_{\bar\Psi}\|_{L^{q'}\ra L^{q'}}\|M_{\bar\Phi}\|_{L^p\ra L^p}, 
\end{equation}
where $\Bar{\Psi}\in B_{q'}$, $\Bar{\Phi}\in B_p$, and
\begin{equation}\label{eqn:twosideconstant}
[u,\sigma]_{A_{p,q,\Psi,\Phi}^\al}=\sup_Q |Q|^{\frac{\al}{n}+\frac1q-\frac1p}
\|u^{\frac1q}\|_{\Psi,Q}\|\sigma^{\frac1{p'}}\|_{\Phi,Q}<\infty.\end{equation}
By H\"older's inequality for Orlicz norms we
have that this quantity is (up to a constant) larger than the
right-hand side of~\eqref{eqn:strongbump1}.

\medskip

As a corollary to Theorem~\ref{thm:Maltwoweight} we can give an
alternative proof of inequality \eqref{eqn:maximalbound2}, which,
again by Theorem~\ref{thm:main}, implies inequalities
\eqref{eqn:weakRP2} and \eqref{eqn:strongRP2}.  We briefly sketch the
argument.  If $(u,\sigma)\in A_{p,q}^\al$ and $\sigma\in
A_\infty$, then Theorem 2.3 in \cite{hytonen-perez-analPDE}
$$\Big(\,\dashint_Q \sigma^r\,dx\Big)^{1/r}\leq 2\dashint_Q \sigma\,dx,$$
where 
$$r=r(\sigma)=1+\frac{1}{c_n[\sigma]_{A^M_\infty}}.$$
Notice that $r'\simeq [\sigma]_{A^M_\infty}$.  Let
$\Phi(t)=t^{rp'}$; then 
$\bar{\Phi}(t)=t^{(rp')'}$ and 
$$[u,\sigma]_{A_{p,q,\Phi}^\alpha}\lesssim [u,\sigma]_{A^\al_{p,q}}.$$
Further, as we noted above
$$\|M_{\beta,\bar{\Phi}}\|_{L^p\ra L^q}\lesssim (r')^{\frac1q}\simeq [\sigma]_{A^M_\infty}^{\frac1q}.$$

\begin{remark}
If we use the original inequality \eqref{thm:perez-bump} in this
argument,  we get a worse power of $1/p$ on the constant $[\sigma]_{A^M_\infty}$:
$$\|M_\al(\,\cdot\,\sigma)\|_{L^p(\sigma)\ra L^q(u)}\lesssim [u,\sigma]_{A^\al_{p,q}}[\sigma]_{A^M_\infty}^{\frac1p}.$$
\end{remark}

\medskip

Theorems \ref{weakbump} and \ref{strongbump} give positive answers
for all  $1<p<q<\infty$ to 
two conjectures we originally made in~\cite{cruz-moen-2012}.  There we
proved  partial results  using a more complicated corona decomposition
argument.  We were forced to assume that $\Phi$ and $\Psi$ were log
bumps: i.e.,
\[ \Phi(t) = t^{p'}\log(e+t)^{p'-1+\delta}, \quad \Psi(t)=
t^q\log(e+t)^{q-1+\delta}, \quad \delta >0, \]
 and make
the further restriction that $(p'/q')(1-\alpha/n)\geq 1$ for the weak
type inequality and $\min(q/p,p'/q')(1-\alpha/n) \geq 1$ for the
strong type inequality.  These conditions hold if $p$ and $q$ satisfy
the Sobolev relationship but do not hold if $p$ and $q$ are very close in
value.    

In ~\cite{cruz-moen-2012} we also conjectured that these results hold
in the critical exponent case $p=q$.  This case is important for its
applications in the study of partial differential equations:
see~\cite{sawyer-wheeden92} and the references it contains.  We repeat
these conjectures here.

\begin{conjecture}  \label{weakpconj} 
Given $0<\al<n$, $1<p<\infty$, and $\bar{\Psi}\in B_{p'}$, suppose
$(u,\sigma)\in A_{p',p',\Psi}^\alpha$.  Then
$$\|I_\al(\,\cdot\,\sigma)\|_{L^p(\sigma)\ra L^{p,\infty}(u)}\lesssim
[\sigma,u]_{A_{p',p',\Psi}^\alpha}\|M_{\bar\Psi}\|_{L^{p'}\ra L^{p'}}.$$
\end{conjecture}

\begin{conjecture}  \label{strongpconj} 
Given $0<\al<n$, $1<p<\infty$, $\bar{\Psi}\in B_{p'}$ and
$\bar{\Phi}\in B_p$, suppose the pair  $(u,\sigma)$ satisfies
$(u,\sigma)\in A_{p',p',\Phi}^\alpha$ and $(\sigma, u)\in A_{p,p,\Psi}^\alpha$
then 
\[ \|I_\al(\,\cdot\,\sigma)\|_{L^p(\sigma)\ra L^{p}(u)}\lesssim
[u,\sigma]_{A_{p,p,\Phi}^\alpha}\|M_{\bar\Phi}\|_{L^{p}\ra L^{p}}
 +[\sigma,u]_{A_{p',p',\Psi}^\alpha}\|M_{\bar\Psi}\|_{L^{p'}\ra L^{p'}}. \]
\end{conjecture}

Very little is known about these conjectures.   We do have that
Conjecture~\ref{weakpconj} implies Conjecture~\ref{strongpconj}, since
for all pairs $(u,\sigma)$ and exponents $1<p\leq q<\infty$, 
\begin{equation} \label{eqn:strong-weak}
\|I_\al(\,\cdot\,\sigma)\|_{L^p(\sigma)\ra L^q(u)}\simeq \|I_\al(\,\cdot\,\sigma)\|_{L^p(\sigma)\ra L^{q,\infty}(u)}+\|I_\al(\,\cdot\,u)\|_{L^{q'}(u)\ra L^{p',\infty}(\sigma)}.
\end{equation}
(See~\cite{sawyer-wheeden92}.)   Conjecture~\ref{weakpconj} is known
in the special case $\Psi(t)= t^p\log(e+t)^{2p-1+\delta}$:  this was
proved in~\cite[Theorem~9.42]{MR2797562}.   Note that the exponent is much larger
than desired: in the case of log bumps we would expect the exponent to
be $p-1+\delta$.    

\begin{remark}
  Conjecture~\ref{strongpconj} is the fractional version of the separated
  bump conjecture for Calder\'on-Zygmund operators made in
  \cite{Hy2012}:
\begin{equation}\label{eqn:sepbumpsi} 
\|T(\,\cdot\,\sigma)\|_{L^p(\sigma)\ra L^p(u)}
\lesssim [u,\sigma]_{A_{p,\Phi}}\|M_{\bar{\Phi}}\|_{L^p\ra L^p}
+ [\sigma,u]_{A_{p',\Psi}}\|M_{\bar{\Psi}}\|_{L^{p'}\ra
  L^{p'}}\end{equation}
(where $[u,\sigma]_{A_{p,\Phi}}=[u,\sigma]_{A_{p,p,\Phi}^0}$).  A
non-quantitative version of this conjecture first appeared
in~\cite{CRV2012}. In this paper they gave a partial result in the
scale of log bumps:  if
\[ 
\Phi(t)=t^{p'}\log(e+t)^{p'-1+\delta},
\Psi(t)=t^p\log(e+t)^{p-1+\delta}, 
\qquad \delta>0,
\]
then
$$\|T(\,\cdot\,\sigma)\|_{L^p(\sigma)\ra L^p(u)}
\lesssim [u,\sigma]_{A_{p,\Phi}}\|M_{\bar{\Phi}}\|^{p+1}_{L^p\ra L^p}
+ [\sigma,u]_{A_{p',\Psi}}\|M_{\bar{\Psi}}\|^{p'+1}_{L^{p'}\ra L^{p'}}.$$
\end{remark}

\medskip

We conclude this section with an observation.  We suspect that the
following result, which gives a connection between operator norms for
the Riesz potential and a ``bilinear'' (properly, bisublinear) maximal
operator defined by the second author in~\cite{MR2514845}, may be
applicable to this problem.  

\begin{theorem} \label{thm:bilinear} 
Given $0< \al<n$ and a dyadic
  grid $\D$, let $X$ and $Y$ be Banach function spaces.  Then
$$\|I_\al\|_{X\ra Y}\simeq \|\M_\al\|_{X\times Y'\ra L^1}, $$
where for $f,\,g\in L^1_{\text{loc}}$, 
$$\M_\al(f,g)(x)=\sup_{Q\ni x} |Q|^{\frac{\al}{n}}\dashint_Q |f|\,dx
\cdot \dashint_Q |g|\,dx.$$
\end{theorem}

Earlier, related estimates for singular integrals were implicit
in~\cite{cruz-uribe-martell-perezP} and the corresponding version of
Theorem \ref{thm:bilinear} for Calder\'on-Zygmund operators was proved
in~\cite{lerner-IMRN2012}. Theorem~\ref{thm:bilinear} is proved in
essentially the same way and we omit the details.  Inequality
\eqref{eqn:dbump} follows from Theorem \ref{thm:bilinear} and the
weighted theory for $\M_\al$ developed in \cite[Theorem
6.6]{MR2514845}, but we are unable to prove separated bump results
using this approach.

\section{Proof of Theorem \ref{thm:main}}
\label{section:main-thm}

We divide this section into two parts.  In the first we gather some
results about dyadic Riesz potentials, and in the second give the
proof itself.  

\subsection*{Dyadic Riesz potentials} 
A dyadic grid, usually denoted $\D$, is a collection of cubes in $\R^n$ with the following properties:
\begin{enumerate}[\indent (a)]
\item given $Q\in \D$, the side-length satisfies $\ell(Q)=2^k$ for some $k\in \Z$;
\item given $Q,P\in \D$, $Q\cap P$ is either $P$, $Q$, or $\varnothing$;
\item for a fixed $k\in \Z$ the set $\D_k=\{Q\in \D:\ell(Q)=2^k\}$ is a partition of $\R^n$.
\end{enumerate}
Given $t\in \{0,1/3\}^n$ we define the family of dyadic grids
$$\D^t=\{2^{-k}([0,1)^n+m+(-1)^kt):k\in \Z, m\in \Z^n\}.$$
When $t=0$,  $\D^0$ is the classic dyadic grid with base point at the
origin used in the Calder\'on-Zygmund decomposition.  

Given a dyadic grid $\D$ and $0<\alpha<n$, we define a dyadic version
of $I_\al$: 
\begin{equation} \label{eqn:dR-defn}
I^{\D}_\al f(x) =\sum_{Q\in \D} \frac{1}{|Q|^{1-\frac{\al}{n}}}\int_Q f(y)\,dy\cdot\chi_Q(x).
\end{equation}
In \cite{cruz-moen-2012} we showed that for non-negative functions $f$, 
\begin{equation}\label{dyadicboundI} 
I_\al f(x) \lesssim \max_{t\in \{0,1/3\}^n} I_\al^{\D^t}f(x).
\end{equation}
Since $I_\alpha$ and $I_\alpha^\D$ are
positive operators, hereafter we may assume that we are dealing with
non-negative functions and can apply these inequalities to reduce to
the dyadic case.  

To estimate the norm of $I_\alpha^\D$, we will use a testing condition
due to Lacey, Sawyer and Uriarte-Tuero~\cite{LacSawUT2010}. To state
their result, we need two definitions.  First, given a cube $Q_0\in \D$, for
$x\in Q_0$ define the ``outer'' dyadic Riesz potential
$$I_\al^{Q_0}f(x)=\sum_{\substack{Q\in \D\\ Q\supset Q_0}} |Q|^{\frac{\al}{n}}\dashint_Q f(y)\,dy\cdot\chi_Q(x).$$
Second, given $0<\alpha<n$,  $1<p<q<\infty$  and a pair of weights
$(u,\sigma)$,  define the testing constant
$$[u,\sigma]_{I_\al^{\text{out}},p,q}=
\sup_{Q_0}\left(\int_{\R^n}I^{Q_0}_\al(\sigma\chi_{Q_0})(x)^{q}
  u\,dx\right)^{1/q}\sigma(Q_0)^{-1/p}.$$ 

\begin{theorem}\label{outtest} Given $0<\alpha<n$ and $1<p<q<\infty$, 
$$\|I^\D_\al(\,\cdot\,\sigma)\|_{L^p(\sigma)\ra L^{q,\infty}(u)}\simeq [\sigma,u]_{I_\al^{\text{out}},q',p'}$$
and
$$\|I^\D_\al(\,\cdot\,\sigma)\|_{L^p(\sigma)\ra L^{q}(u)}\simeq
[u,\sigma]_{I_\al^{\text{out}},p,q}
+[\sigma,u]_{I_\al^{\text{out}},q',p'}.$$
\end{theorem}

\begin{remark}
In Theorem~\ref{outtest} the restriction that $p<q$ is essential; this
is the reason for this condition for our paper.
In~\cite{LacSawUT2010} they give a different testing condition that
holds when $p=q$, but we have been unable to apply our techniques to
get estimates in this case.
\end{remark}

\subsection*{Proof of Theorem~\ref{thm:main}}
Our argument is broadly similar to the one
in~\cite{cruz-martell-perezP}.  By inequality~\eqref{dyadicboundI} it
suffices to fix a dyadic grid $\D$ and obtain norm estimates for
$I_\alpha^\D$ that are independent of the grid.  And by Theorem
\ref{outtest} it suffices to estimate the testing constant for the
outer Riesz potential.  The inequality ``$\lesssim"$ in Theorem
\ref{thm:main} is a consequence of the following result.

\begin{theorem} 
Given $0<\al<n$, $1<p<q<\infty$ and a pair of weights $(u,\sigma)$, then
$$[u,\sigma]_{I_\al^{\text{out}},p,q}\leq  
(1-2^{\al-n})^{-1} \|M_\alpha(\,\cdot\,\sigma)\|_{L^{p}(\sigma)\ra L^{q}(u)}$$
and
$$[\sigma,u]_{I_\al^{\text{out}},q',p'}\leq 
(1-2^{\al-n})^{-1} \|M_\alpha(\,\cdot\,u)\|_{L^{q'}(u)\ra L^{p'}(\sigma)}.$$
\end{theorem}

\begin{proof} 
We will prove the first inequality; the proof of the second is
identical.  
Fix a cube $Q_0\in \D$ and for each $k\geq1$ let $Q_k \in\D$ be the
unique cube such that $Q_{k-1}\subset Q_k$ and  $|Q_k|=2^{kn}|Q_0|$.
By definition,
$$I_\al^{Q_0}(\sigma\chi_{Q_0})(x)
=\sum_{Q\supset Q_0} |Q|^{\frac{\al}{n}-1}\int_Q \sigma\chi_{Q_0}\,dx\cdot
\chi_Q(x)
=\int_{Q_0}\sigma\,dx\cdot\sum_{k=0}^\infty|Q_k|^{\frac{\al}{n}-1}\cdot\chi_{Q_k}(x).$$
Clearly, the support of $I_\al^{Q_0}(\sigma\chi_{Q_0})$ is
$\bigcup_{k=0}^\infty Q_k$.  If $x\in Q_0$, then 
\begin{multline*}
I_\al^{Q_0}(\sigma\chi_{Q_0})(x)=\int_{Q_0}\sigma\,dx\cdot\sum_{k=0}^\infty |Q_k|^{\frac{\al}{n}-1}=|Q_0|^{\frac{\al}{n}-1}\int_{Q_0}\sigma\,dx\cdot\sum_{k=0}^\infty 2^{k(\al-n)}\\
= (1-2^{\al-n})^{-1} |Q_0|^{\frac{\al}{n}-1}\int_{Q_0}\sigma\,dx 
\leq  (1-2^{\al-n})^{-1}M_\al(\sigma\chi_{Q_0})(x).
\end{multline*}
If $x\in Q_{j+1}\backslash Q_j$ for some $j\geq 0$, then
\begin{multline*}
I_\al^{Q_0}(\sigma\chi_{Q_0})(x) =\int_{Q_0}\sigma\,dx\cdot\sum_{k=j}^\infty
|Q_k|^{\frac{\al}{n}-1} \\
 =|Q_0|^{\frac{\al}{n}-1}\int_{Q_0}\sigma\,dx\cdot\sum_{k=j}^\infty 2^{k(\al-n)}
= (1-2^{\al-n})^{-1}2^{j(\al-n)}|Q_0|^{\frac{\al}{n}-1}\int_{Q_0}\sigma\,dx \\
=(1-2^{\al-n})^{-1}|Q_j|^{\frac{\al}{n}-1}\int_{Q_0}\sigma\,dx
\leq (1-2^{\al-n})^{-1} M_\al(\sigma\chi_{Q_0})(x).
\end{multline*}
We therefore have that
$$\int_{\R^n}I_\al^{Q_0}(\sigma\chi_{Q_0})^{q}u\,dx
\le (1-2^{\al-n})^{-q} \int_{\R^n}M_\al(\sigma\chi_{Q_0})^{q} u\,dx,$$
and the desired inequality follows immediately.
\end{proof}

\medskip

Finally, we prove the reverse inequalities in Theorem~\ref{thm:main}.
We will prove that 
\[ \| M_\alpha(\,\cdot\, u)\|_{L^{q'}(u)\ra L^{p'}(\sigma)} \lesssim
\|I_\alpha(\,\cdot\, \sigma )\|_{L^p(\sigma) \ra L^{q,\infty}(u)}; \]
the other estimates are proved in essentially the same way.
Fix a cube $Q$; then 
\begin{multline*} \int_Q I_\alpha (f\sigma) u\,dx 
 \leq \|I_\alpha
(f\sigma)\|_{L^{q,\infty}(u)}\|\chi_Q\|_{L^{q',1}(u)} \\
 \leq \|I_\alpha(\cdot \sigma) \|_{L^p(\sigma) \ra L^{q,\infty}(u)}
\|f\|_{L^p(\sigma)}u(Q)^{1/q'}.
\end{multline*}
Let $f= I_\alpha(u\chi_Q)^{p'-1}\chi_Q$.  Then, since $I_\alpha$ is
self-adjoint, we have that 
\[ \left(\int_Q I_\alpha(u\chi_Q)^{p'} \sigma\,dx\right)^{1/p'} \leq
\|I_\alpha(\,\cdot\, \sigma )\|_{L^p(\sigma) \ra L^{q,\infty}(u)} u(Q)^{1/q'}. \]
Since for non-negative functions $f$, $M_\alpha f(x) \lesssim I_\alpha
f(x)$, we have that

\begin{multline*} [\sigma,u]_{M_\alpha,q',p'} = \sup_Q 
\left(\int_Q M_\alpha(u\chi_Q)^{p'}
  \sigma\,dx\right)^{1/p'}u(Q)^{-1/q'}
\\ \lesssim\|I_\alpha(\cdot \sigma) \|_{L^p(\sigma) \ra L^{q,\infty}(u)}. \end{multline*}

But by Sawyer's testing condition for the fractional maximal
operator~\cite{sawyer82b},
\[ \| M_\alpha(\,\cdot\, u)\|_{L^{q'}(u)\ra L^{p'}(\sigma)} \simeq
[\sigma,u]_{M_\alpha,q',p'}. \]
This completes the proof.

\section{Estimates involving $A_\infty$}
\label{section:other-proofs}

In this section we prove
Theorems~\ref{maximalbound},~\ref{thm:red-thm}
and~\ref{thm:Mlog-Ainfty}.  We first give some further results on
dyadic operators, and we then prove each of these theorems in turn.

\subsection*{Dyadic fractional maximal operators}

Given a dyadic grid $\D$, define the dyadic fractional maximal
operator by
$$M^\D_\al f(x) =\sup_{Q\in \D} \frac{1}{|Q|^{1-\frac{\al}{n}}}\int_Q f(y)\,dy\cdot\chi_Q(x).$$
Essentially the same argument in~\cite{cruz-moen-2012} that gave
us~\eqref{dyadicboundI} also lets us prove the corresponding estimate
for the fractional maximal operator:
\begin{equation}\label{dyadicboundM}
M_\al f(x)\lesssim  \max_{t\in \{0,1/3\}^n} M_\al^{\D^t}f(x).
\end{equation}
Therefore, it will suffice to prove norm estimates for $M_\alpha^\D$.
In fact, we will prove estimates for a linearization of this operator.
We begin by defining sparse families of cubes.  These are a
generalization of an idea closely connected to the Calder\'on-Zygmund
decomposition.  Given a dyadic grid $\D$, a subset $\Sp\subseteq \D$
is sparse if there exists a family of disjoint, ``thick''  subsets: for
every $Q\in \Sp$ there exists $E_Q\subseteq Q$ such that the family
$\{E_Q\}_{Q\in \Sp}$ is pairwise disjoint and $|E_Q|\geq \frac12|Q|$.

Given a sparse family $\Sp$ define the linear operator
$$L_\al^\Sp f(x)=\sum_{Q\in \Sp} |Q|^{\frac{\al}{n}}\dashint_Q f\,dx
\cdot\chi_{E_Q}(x).$$

\begin{theorem} \label{thm:sparse-max}
Given $0\leq \al<n$ and a dyadic grid $\D$,  for every non-negative
Function $f \in L^\infty_c(\R^n)$ there exists a sparse family
$\Sp\subset \D$ such that for a.e. $x$,
\begin{equation} \label{eqn:sawyer-ptwise}
 M_\al^\D f(x)\lesssim L^\Sp_\al f (x),
\end{equation}
and the implicit constants do not depend on $\D$, $\Sp$ or $f$.  
Consequently,  for any Banach function spaces $X$ and $Y$
$$\|M_\al^\D\|_{X\ra Y}\simeq \sup_{\Sp}\|L^\Sp_\al\|_{X\ra Y}$$
and
$$\|M_\al\|_{X\ra Y}\simeq \sup_{\Sp,\D}\|L^\Sp_\al\|_{X\ra Y}.$$
\end{theorem}

The pointwise inequality~\eqref{eqn:sawyer-ptwise} is well known; the
basic idea first appeared in Sawyer~\cite{sawyer82b} (see
also~\cite[Section~9.3]{MR2797562}).  The norm inequalities follow
from this and~\eqref{dyadicboundM}.

\subsection*{Proof of Theorem \ref{maximalbound}}
For the proof we need three lemmas.  The first is an $L^p$ estimate
for a dyadic maximal operator defined with respect to an arbitrary
measure.  Let $\mu$
be a positive Borel measure on $\R^n$;  given a dyadic grid $\D$ and $0\leq\beta<n$, define
$$M^\D_{\beta,\mu}f(x)=\sup_{Q\in \D}
\frac{1}{\mu(Q)^{1-\frac{\beta}{n}}}\int_Q |f|\,d\mu \cdot
\chi_Q(x).$$
When $\beta=0$ we simply write $M_{0,\mu}^\D=M_\mu^\D$.   The operator
$M^\D_\beta$ satisfies norm inequalities with bounds independent of
$\mu$:  the proof of the next result can be found in~\cite{moen-2012}. 

\begin{lemma}\label{weightedmax} Given $0\leq \beta<n$ and $1<p\leq
  \frac{n}{\beta}$, 
define $\frac{1}{q}=\frac{1}{p}-\frac{\beta}{n}$.  If $\mu$ is a
measure such that $\mu(\R^n)=\infty$,  then
$$ \|M^\D_{\beta,\mu}\|_{L^p(\mu)\ra L^q(\mu)} \leq \Big(1+\frac{p'}{q}\Big)^{1-\frac{\beta}{n}}.$$
\end{lemma}

The second lemma is a weighted Carleson embedding theorem.  To state
it we make two definitions.   Given a
sequence of positive numbers ${\bf c}=\{c_Q\}_{Q\in \D}$, we say that
${\bf c}$ is a Carleson sequence with respect to a measure $\mu$ if
for each $Q_0\in \D$,
\begin{equation}\label{carleson} 
\sum_{Q\subseteq Q_0} c_Q\leq C \mu(Q_0).
\end{equation}
The infimum of the constants in inequality \eqref{carleson} will be
denoted $C({\bf c})$.   Also, given a sequence ${\bf  a}=\{a_Q\}_{Q\in \D}$, define the sequential
maximal operator
$$M{\bf a}(x)=\sup_{Q\in \D} |a_Q|\cdot\chi_Q(x).$$

The proof of the next result is also standard;  for instance,
see~\cite[Theorem~4.5]{hytonen-perez-analPDE}. 

\begin{lemma}  \label{carlesonlem} 
Fix $1<p<\infty$ and a dyadic grid $\D$.  If  ${\bf a}=\{a_Q\}_{Q\in \D}$ is any sequence and
${\bf c}=\{c_Q\}_{Q\in \D}$ is a Carleson sequence with respect to a
measure $\mu$, then
$$\sum_{Q\in \D} |a_Q|^p\,c_Q\leq C({\bf c}) \int_{\R^n} M{\bf a}(x)^p\,d\mu.$$
\end{lemma}

Finally, define the geometric maximal operator by
$$M^\D_0f(x)=\sup_{Q\in \D} \exp\Big(\dashint_Q \log
|f|\,dx\Big)\cdot \chi_Q(x).$$
By Jensen's inequality, for any $r>0$,  $M^\D_0f(x)\leq M^\D (|f|^r)(x)^{1/r}$, so the geometric
maximal operator is bounded on $L^p$, $p>0$.  The sharp constant can
be readily computed using this fact:  see, for
instance,~\cite[Lemma~2.1]{hytonen-perez-analPDE}.  (For the history
of this operator, see~\cite{MR1628101} and the references it contains.)

\begin{lemma} \label{geometric} Given $0<p<\infty$,
$$\|M^\D_0 f\|_{L^p(\R^n)}\leq e\|f\|_{L^p(\R^n)}.$$
\end{lemma}

\bigskip

\begin{proof}[Proof of Theorem \ref{maximalbound}] 
Fix a non-negative function $f$.  
By Theorem~\ref{thm:sparse-max} it will  suffice to get a norm estimate
  for the linearization 
$$L^\Sp_\al(f\sigma)=\sum_{Q\in \Sp} |Q|^{\frac{\al}{n}}\dashint_Q
f\sigma\,dx\cdot \chi_{E_Q},$$
with a constant that is independent of the sparse subset
$\Sp\subset\D$.  
Since the sets $\{E_Q\}_{Q\in \Sp}$ are pairwise disjoint,
\begin{multline} \label{eqn:first-step}
\int_{\R^n}L^\Sp_\al(f\sigma)^qu\,dx
=\sum_{Q\in \Sp} \left(|Q|^{\frac{\al}{n}}\dashint_Q f\sigma\,dx\right)^{q}u(E_Q)\\
\leq\sum_{Q\in \Sp} \left(|Q|^{\frac{\al}{n}
+\frac1q-\frac{1}{p}}\Big(\dashint_Qu\,dx\Big)^{1/q}
\Big(\dashint_Q\sigma\,dx\Big)^{1/p'}\int_Q f\sigma\,dx\right)^{q}\sigma(Q)^{-q/p'}.
\end{multline}
Now let $\beta=n\Big(\frac1p-\frac1q\Big)$; then
$$\sigma(Q)^{-q/p'}=\Big(\frac{1}{\sigma(Q)^{1-\frac{\beta}{n}}}\Big)^q\sigma(Q).$$
Define the sequences ${\bf c}=\{c_Q\}$ and ${\bf a}=\{a_Q\}$ by 
$$c_Q=\left(|Q|^{\frac{\al}{n}+\frac1q-\frac{1}{p}}\Big(\dashint_Qu\,dx\Big)^{1/q}\Big(\dashint_Q\sigma\,dx\Big)^{1/p'}\right)^{q}\sigma(Q)$$
and
$$a_Q=\frac{1}{\sigma(Q)^{1-\frac{\beta}{n}}}\int_Qf\sigma\,dx$$
for $Q\in \Sp$ and zero otherwise.
Then we can rewrite \eqref{eqn:first-step} as
\begin{equation}
\label{mainest}\|L^\Sp_\al(f\sigma)\|_{L^q(u)}\leq  \Big(\sum_{Q\in \Sp}(a_Q)^qc_Q \Big)^{1/q}.
\end{equation}

To estimate the right-hand side of \eqref{mainest} we will show that
${\bf c}$ is a Carleson sequence with respect to the
measure $\sigma$ and its Carleson constant is bounded by either
\begin{equation}\label{Carlesonest1}
C({\bf c})\leq
2e[u,\sigma]^q_{A_{p,q}(u,\sigma)A^{\exp}_\infty(\sigma)^{\frac1q}}
\end{equation}
or
\begin{equation}\label{Carlesonest2}
C({\bf c}) \leq 2[u,\sigma]^q_{A_{p,q}}[\sigma]_{A^M_\infty}.
\end{equation}
Given these estimate we are done: by Lemma \ref{carlesonlem} we have
$$ \|L_\al^\Sp(f\sigma)\|_{L^q(u)}\leq C({\bf c})^{1/q}\|M{\bf
  a}\|_{L^q(\sigma)}.$$
By our choice of the sequence ${\bf a}$,
$M{\bf a}(x)=M_{\beta,\sigma}^\D f(x).$
Since $\sigma\in A_\infty$, $\sigma(\R^n)=\infty$, so by Lemma
\ref{weightedmax},
$$\|L_\al^\Sp(f\sigma)\|_{L^q(u)}\leq 
C({\bf c})^{1/q}\|M_{\beta,\sigma}^\D f\|_{L^q(\sigma)}
\lesssim C({\bf c})^{1/q}\|f\|_{L^p(\sigma)}.$$

\bigskip

To complete the proof we first  prove \eqref{Carlesonest1}.  Let $Q_0\in\D$; then by  Lemma \ref{geometric},
\begin{align*}
\sum_{Q\subseteq Q_0} c_Q
&=\sum_{\stackrel{Q\in \Sp}{Q\subseteq Q_0}} \left(|Q|^{\frac{\al}{n}+\frac1q-\frac{1}{p}}\Big(\dashint_Qu\,dx\Big)^{1/q}\Big(\dashint_Q\sigma\,dx\Big)^{1/p'}\right)^{q}\sigma(Q)\\
&\leq 2[u,\sigma]^q_{A_{p,q}^\alpha(u,\sigma)A^{\exp}_\infty(\sigma)^{\frac1q}}
\sum_{\substack{Q\in \Sp \\ Q\subseteq Q_0}} \exp\Big(\dashint_Q \log \sigma\Big)|E_Q|\\
&\leq 2[u,\sigma]^q_{A_{p,q}^\alpha(u,\sigma)A^{\exp}_\infty(\sigma)^{\frac1q}}\int_{Q_0}M_0(\chi_{Q_0}\sigma)\,dx\\
&\leq 2e[u,\sigma]^q_{A_{p,q}^\alpha (u,\sigma)A^{\exp}_\infty(\sigma)^{\frac1q}}\sigma(Q_0).
\end{align*}

We prove  \eqref{Carlesonest2} in a similar fashion:  by the
definition of the $A_\infty^M$ constant, 
\begin{align*}
\sum_{Q\subseteq Q_0} c_Q&=\sum_{\stackrel{Q\in \Sp}{Q\subseteq Q_0}} \left(|Q|^{\frac{\al}{n}+\frac1q-\frac{1}{p}}\Big(\dashint_Qu\,dx\Big)^{1/q}\Big(\dashint_Q\sigma\,dx\Big)^{1/p'}\right)^{q}\sigma(Q)\\
&\leq 2[u,\sigma]_{A_{p,q}^\alpha}^q\sum_{\stackrel{Q\in \Sp}{Q\subseteq Q_0}}\frac{\sigma(Q)}{|Q|}|E_Q|\\
&\leq2[u,\sigma]_{A_{p,q}^\alpha}^q\int_{Q_0}M(\chi_{Q_0}\sigma)\,dx\\
&\leq 2[u,\sigma]_{A_{p,q}^\alpha}^q[\sigma]_{A^M_\infty}\sigma(Q_0).
\end{align*}

\end{proof}

\subsection*{Proof of Theorem~\ref{thm:red-thm}}
To apply Theorem~\ref{thm:onewt-weak} we first make a few preliminary
remarks.  Given any $\alpha$ and exponents  $p$ and $q$ that satisfy
the Sobolev relationship, define
\[ s(p) = 1+\frac{p}{q'} = q\left(1-\frac{\alpha}{n}\right).  \]
Given a weight $w$, if we let $u=w^q$ and $\sigma=w^{-p'}$, then it is
immediate that
\[ [ u,\sigma]_{A_{p,q}^\alpha} = [ w^q ]_{A_{s(p)}}^{\frac1q}.  \]
With this notation we can restate~\eqref{eqn:weakRP2} as 
\begin{equation} \label{eqn:onewt-weak}
\|I_\alpha\|_{L^p(w^p)\ra L^{q,\infty}(w^q)} \lesssim
[w^q]_{A_{s(p)}}^{\frac1q}[w^q]_{A_\infty}^{\frac{1}{p'}}.
\end{equation}

Now fix $r$ and $w$ as in the hypotheses.  We will prove \eqref{eqn:red-thm1}
using interpolation with change of measure.   Since $r<s(p)$, there
exist $p_0$ and $q_0$ such that $1/p_0-1/q_0=\alpha/n$ and 
\[  r = s(p_0) = q_0\left(1-\frac{\alpha}{n}\right).  \]
Note that $p_0<p$ and $q_0<q$.   Define $w_0=w^{q/q_0}$;  then 
\[ [w_0^{q_0}]_{A_{s(p_0)}} = [w^q]_{A_r}, \]
and so by \eqref{eqn:onewt-weak},
\begin{equation} \label{eqn:lower}
\|I_\alpha\|_{L^{p_0}(w_0^{p_0})\ra L^{q_0,\infty}(w^q)} \lesssim
[w^q]_{A_r}^{\frac1q_0}[w^q]_{A_\infty}^{\frac1{p_0'}}.
\end{equation}

Define $p_1>p$ and $q_1>q$ by
\[ \frac{1/2}{p_0}+\frac{1/2}{p_1}=\frac{1}{p}, \qquad
\frac{1/2}{q_0}+\frac{1/2}{q_1}=\frac{1}{q}. \]
Let $w_1=w^{q/q_1}$.  Since $s(p_1)>s(p)>r$, we have that 
\[ [ w_1^{q_1}]_{A_{s(p_1)}} \leq [w^q]_{A_r}. \]
Therefore, if we repeat the above argument, we get that 
\begin{equation} \label{eqn:upper}
\|I_\alpha\|_{L^{p_1}(w_1^{p_1})\ra L^{q_1,\infty}(w^q)} \lesssim
[w^q]_{A_r}^{\frac1q_1}[w^q]_{A_\infty}^{\frac1{p_1'}}.
\end{equation}

Given inequalities~\eqref{eqn:lower} and~\eqref{eqn:upper},
by interpolation with change of measure (Stein and
Weiss~\cite{stein-weiss58}; also see
Grafakos~\cite[Exercise~1.4.9]{grafakos08a} for a careful treatment of
the constants) we get~\eqref{eqn:red-thm1}.

\subsection*{Proof of Theorem~\ref{thm:Mlog-Ainfty}} 
Again, by \eqref{dyadicboundM} it suffices to work with the dyadic
operator $M_\al^\D$.  We begin with a testing condition from
\cite{MR2534183}.  Given $p$ and $q$ satisfying the Sobolev
relationship,  define 
$$[u,\sigma]_{M^\D,s(p)}=
\sup_{R\in \D}\frac{\Big(\int_R M^\D(\sigma
  \chi_R)^{s(p)}u\,dx\Big)^{1/q}}
{\sigma(R)^{1/q}}$$
 (recall $s(p)=1+\frac{q}{p'}$).  It was shown in \cite[Corollary 4.5]{MR2534183} that
$$\|M^\D_\al(\,\cdot\,\sigma)\|_{L^p(\sigma)\ra L^q(u)}
\lesssim [u,\sigma]_{M^\D,s(p)};$$ 
that is, the two weight norm
inequality of the fractional maximal operator is bounded by the
testing constant of the Hardy-Littlewood maximal operator.  

Now fix $w^q\in A_{s(p)}$ and let $u=w^q$ and $\sigma=w^{-p'}$.  We
will estimate $[u,\sigma]_{M^\D,s(p)}$.  Fix $R\in\D$.  By
inequality \eqref{eqn:sawyer-ptwise} there exists a sparse family
$\Sp\subset R$ such that
\[
\int_R M^\D(\sigma\chi_Q)^{s(p)}u\,dx
 \simeq \sum_{Q\in \Sp}\Big(\frac{\sigma(Q)}{|Q|}\Big)^{s(p)}u(E_Q)
\leq\sum_{Q\in \Sp}\Big(\frac{\sigma(Q)}{|Q|}\Big)^{1+\frac{q}{p'}}
\frac{u(Q)}{|Q|}|Q|.
\]
 For $a\in \Z$ define
 $$\Sp^a=\Big\{Q\in \Sp: 2^{a}< \Big(\frac{u(Q)}{|Q|}\Big)^{\frac{p'}{q}}\frac{\sigma(Q)}{|Q|}\leq 2^{a+1}\Big\}.$$
 Since $w^q\in A_{s(p)}$, the sets $\Sp^a$ are empty if $a<-1$ or
 $a> \lfloor\log_2[w^{-p'}]_{A_{s(q')}}\rfloor:=K$.  (The fact that
 $u=w^q$ and $\sigma=w^{-p'}$ is essential at this step.)  
Hence,
 \begin{multline}\sum_{Q\in
     \Sp}\Big(\frac{\sigma(Q)}{|Q|}\Big)^{1+\frac{q}{p'}}\frac{u(Q)}{|Q|}|Q| \\
=\sum_{a=-1}^K\sum_{Q\in \Sp^a}
\Big(\frac{\sigma(Q)}{|Q|}\Big)^{1+\frac{q}{p'}}\frac{u(Q)}{|Q|}|Q|
 \label{sumest}\simeq \sum_{a=-1}^K2^{a\frac{q}{p'}}
\sum_{Q\in \Sp^a}\Big(\frac{\sigma(Q)}{|Q|}\Big)|E_Q|.
 \end{multline}

 We now analyze the inner sum in \eqref{sumest}.  For each $a$,
 $-1\leq a\leq K$, let $\Sp_{\max}^a$ be the collection of maximal
 cubes with respect to inclusion in $\Sp^a$.  Then the family
 $\Sp^a_{\max}$ is pairwise disjoint and every cube in $\Sp^a$ is
 contained in a cube from $\Sp^a_{\max}$.  Thus,
\[ 
\sum_{Q\in \Sp^a}\Big(\frac{\sigma(Q)}{|Q|}\Big)|E_Q|
=\sum_{Q\in \Sp_{\max}^a}\sum_{\substack{P\in \Sp^a \\ P\subset Q}}\Big(\frac{\sigma(P)}{|P|}\Big)|E_P|
\leq \sum_{Q\in \Sp_{\max}^a}\int_Q M(\sigma\chi_Q)\,dx.
\]
If we substitute this estimate into \eqref{sumest}, then we have that
 \begin{align*}
 \sum_{a=-1}^K2^{a\frac{q}{p'}}\sum_{Q\in \Sp^a}\Big(\frac{\sigma(Q)}{|Q|}\Big)|E_Q|&\leq  \sum_{a=-1}^K2^{a\frac{q}{p'}}\sum_{Q\in \Sp_{\max}^a}\int_Q M(\sigma\chi_Q)\,dx\\
&\leq \sum_{a=-1}^K\sum_{Q\in \Sp_{\max}^a}\Big(\frac{\sigma(Q)}{|Q|}\Big)^{\frac{q}{p'}}\frac{u(Q)}{|Q|}\int_Q M(\sigma\chi_Q)\,dx\\
&\leq [w^{-p'}]_{(A_{s(q')})^{\frac1{p'}}(A^M_\infty)^{\frac1q}}^q\sum_{a=-1}^K\sum_{Q\in \Sp_{\max}^a}\sigma(Q)\\
&\leq (2+K)[w^{-p'}]_{(A_{s(q')})^{\frac1{p'}}(A^M_\infty)^{\frac1q}}^q\sigma(R).
\end{align*}
 If we combine the above inequalities, we get the desired estimate.

 \section{Two weight $A_p$ bump conditions}
\label{section:op2}

In this section we construct Example~\ref{example:bad-wt} and prove
Theorems~\ref{thm:Mfracorlicz} and~\ref{thm:Maltwoweight}.

\subsection*{Construction of Example~\ref{example:bad-wt}}
To construct the desired example, we need to consider two cases.  In
both cases we will work on the real line. 

The simpler case is if $\frac{1}{p}-\frac{1}{q}>\frac{\alpha}{n}$.
Note that in this case, by the Lebesgue differentiation theorem, if
$(u,\sigma)\in A_{p,q}^\alpha$, then $u$ and $\sigma$ have disjoint
supports.  Let $f=\sigma=\chi_{[-2,-1]}$ and let
$u=x^t\chi_{[0,\infty)}$, where $t=q(1-\alpha)-1$.  Given any
$Q=(a,b)$, $A_{p,q}^\alpha(u,\sigma,Q)=0$ unless $a<-1$ and $b>0$.  In
this case we have that
\begin{align*}
A^\alpha_{p,q}(u,\sigma,Q) 
& \leq b^{\alpha+\frac1q-\frac1p}
\left(\frac1b\int_0^b x^t\,dx\right)^{\frac1q}
\left(\frac1b \int_{-2}^{-1} \,dx\right)^{\frac{1}{p'}} \\
& \lesssim b^{\alpha + \frac{t+1}{q}-1} \\
& = 1.
\end{align*}
Hence, $(u,\sigma)\in A_{p,q}^\alpha$. 
On the other hand, for all $x>1$, 
\[ M_\alpha(f\sigma)(x) \approx x^{\alpha-1}, \]
and so
\[ \int_{\R} M_\alpha(f\sigma)(x)^qu(x)\,dx 
\geq  \int_1^\infty x^{q(\alpha-1)}x^{q(1-\alpha)-1}\,dx
= \int_1^\infty \frac{dx}{x} = \infty. \]

\bigskip

Now suppose $\frac{1}{p}-\frac{1}{q}\leq \frac{\alpha}{n}$.  We begin
with a general lemma that lets us construct pairs in $A_{p,q}^\alpha$; this is
an extension of the technique of factored weights developed
in~\cite[Chapter~6]{MR2797562}.

\begin{lemma} \label{lemma:factored}
Given $0<\alpha<n$, suppose $1<p\leq q<\infty$ and
  $\frac{1}{p}-\frac{1}{q}\leq \frac{\alpha}{n}$ .  Let $w_1,\,w_2$ be
  locally integrable functions, and define
\[ u = w_1\big(M_\gamma w_2\big)^{-\frac{q}{p'}}, \qquad
\sigma = w_2\big(M_\gamma w_1\big)^{-\frac{p'}{q}}, \]
where
\[ \gamma = \frac{\frac{\alpha}{n}+\frac{1}{q}-\frac{1}{p}}
{\frac{1}{n}\left(1+\frac{1}{q}-\frac{1}{p}\right)}. \]
Then $(u,\sigma)\in A^\alpha_{p,q}$ and $[u,\sigma]_{A^\alpha_{p,q}}\leq 1$. 
\end{lemma}

\begin{proof}
By our assumptions on $p$, $q$ and $\alpha$, $0\leq \gamma \leq \alpha$.  
Fix a cube $Q$.  Then 
\begin{align*}
A_{p,q}^\alpha(u,\sigma,Q)
& =  |Q|^{\frac{\alpha}{n}+\frac{1}{q}-\frac{1}{p}}
\left(\dashint_Q w_1\big(M_\gamma w_2\big)^{-\frac{q}{p'}}\,dx\right)^\frac{1}{q}
\left(\dashint w_2\big(M_\gamma w_1\big)^{-\frac{p'}{q}} \,dx\right)^\frac{1}{p'} \\
& \leq  |Q|^{\frac{\alpha}{n}+\frac{1}{q}-\frac{1}{p}}
\left(\dashint w_1\,dx\right)^\frac{1}{q'}
\left(|Q|^\frac{\gamma}{n} \left(\dashint_Q w_2\,dx\right)\right)^{-\frac{1}{p'}} \\
& \qquad \qquad \times 
\left(\dashint w_2\,dx\right)^\frac{1}{p'}
\left(|Q|^\frac{\gamma}{n} \left(\dashint_Q w_1\,dx\right)\right)^{-\frac{1}{q}} \\
& =  |Q|^{\frac{\alpha}{n}+\frac{1}{q}-\frac{1}{p}-\frac{\gamma}{n}\left(1+\frac{1}{q}-\frac{1}{p}\right)} \\
& = 1. 
\end{align*}
\end{proof}

With $n=1$, fix $\gamma$ as in Lemma~\ref{lemma:factored}.  The second
step is to construct a set $E\subset [0,\infty)$ such that
$M_\gamma(\chi_E)(x)\approx 1$ for $x>0$.  Let
\[ E = \bigcup_{j\geq 0} [j,j + (j+1)^{-\gamma}).  \]
Suppose $x\in [k,k+1)$; if  $k=0$, then it is immediate that if we take $Q=[0,2]$, then 
$M_\gamma(\chi_E) \geq 3\cdot 2^{\gamma-2}\approx 1$.  If 
$k\geq 1$, let $Q=[0,x]$; then 
\begin{multline*} M_\gamma(\chi_E)(x) \geq 
x^{\gamma-1}\sum_{0\leq j \leq \lfloor x \rfloor} (j+1)^{-\gamma}
\geq (k+1)^{\gamma-1} \sum_{j=0}^{k} (j+1)^{-\gamma} \\ \simeq (k+1)^{\gamma-1} (k+1)^{1-\gamma} = 1. \end{multline*}

\medskip

It remains to prove the reverse inequality.  If $|Q|\leq 1$, then 
\[ |Q|^{\gamma-1}|Q\cap E| \leq |Q|^{\gamma} \leq 1, \]
so we only have to consider $Q$ such that $|Q|\geq 1$.    In this case, given a $Q$, let $Q'$ be the smallest interval whose endpoints are integers that contains $Q$.  Then $|Q'|\leq |Q|+2 \leq 3|Q|$, and so  $|Q|^{\gamma-1}|E\cap Q|\approx |Q'|^{\gamma-1}|E\cap Q'|$.  Therefore, without loss of generality, it suffices to consider $Q=[a,a+h+1]$, $a,\,h$ non-negative integers. Then
\begin{multline*}
 |Q|^{\gamma-1}|Q\cap E| = (1+h)^{\gamma-1} \sum_{a\leq j \leq a+h} (j+1)^{-\gamma}
\approx (1+h)^{\gamma-1} \int_a^{a+h} (t+1)^{-\gamma}\,dt \\
\approx (1+h)^{\gamma-1}\big( (a+h+1)^{1-\gamma}-(a+1)^{1-\gamma}\big).
\end{multline*}
To estimate the last term suppose first that $h \leq a$.  Then by the mean value theorem the last term is dominated by 
\[  (1+h)^{\gamma-1}(1+h)(a+1)^{-\gamma}  \leq 1. \]
On the other hand, if $h>a$, then the last term is dominated by 
\[ (1+h)^{\gamma-1}(a+h+1)^{1-\gamma} \leq 2^{1-\gamma} \approx 1. \]

\bigskip

We can now give our desired counter example.   Let $w_1=\chi_E$ and let $w_2=\chi_{[0,1]}$.   Then for all $x\geq 2$, 
\[ M_\gamma w_1(x) \approx 1, \qquad 
M_\gamma w_2(x) =\sup_{Q} |Q|^{\gamma-1}\int_Q w_2\,dt \approx x^{\gamma-1}. \]
Then by Lemma~\ref{lemma:factored}, if we set
\[ u = w_1 (M_\beta w_2)^{-\frac{q}{p'}}, \qquad
\sigma = w_2 (M_\beta w_1)^{-\frac{p'}{q}},  \]
then $(u,\sigma)\in A^\alpha_{p,q}$.  Moreover, for $x\geq 2$, we have that 
\[ u(x) \approx x^{(1-\gamma)\frac{q}{p'}}\chi_E, \qquad  \sigma(x) \approx \chi_{[0,1]}. \]

Fix $f\in L^p(\sigma)$: without loss of generality, we may assume $\supp(f)\subset [0,1]$.  Then $f\sigma$ is locally integrable, and for $x\geq 2$ we have that
\[ M_\alpha (f\sigma)(x) \geq x^{\alpha -1} \|f\sigma\|_1 \approx x^{\alpha-1}. \]
Therefore, for $x\geq 2$, 
\[ M_\alpha(f\sigma)(x)^q u(x) \gtrsim x^{(\alpha-1)q}x^{(1-\gamma) \frac{q}{p'}}\chi_E. \]
By the definition of $\gamma$, 
\[ \gamma\left(\frac{1}{q}+\frac{1}{p'}\right) = 
\gamma\left(1+\frac{1}{q}-\frac{1}{p}\right) = \alpha+\frac{1}{q}-\frac{1}{p}
= \alpha-1 + \frac{1}{q}+\frac{1}{p'}; \]
equivalently,
\[ (\gamma-1)\left(\frac{q}{p'}+1\right) = q(\alpha-1),\]
and so
\[ (\alpha-1)q+(1-\gamma)\frac{q}{p'} = \gamma-1,  \]

Therefore, to show that $M_\alpha(f\sigma)\not\in L^q(u)$, it will be
enough to prove that
\[ \int_2^\infty x^{\gamma-1}\chi_E(x)\,dx =\infty. \]
This is straight-forward:
\begin{align*}
\int_2^\infty x^{\gamma-1}\chi_E(x)\,dx
 &= \sum_{j=2}^\infty \int_j^{j+(j+1)^{-\gamma}} x^{\gamma-1}\,dx \\
& \geq \sum_{j=2}^\infty (j+(j+1)^{-\gamma})^{\gamma-1} (j+1)^{-\gamma} \\
& \geq \sum_{j=2}^\infty (j+1)^{\gamma-1} (j+1)^{-\gamma} \\
 &\geq \sum_{j=2}^\infty  (j+1)^{-1} = \infty. 
\end{align*}

\subsection*{Proof of Theorems \ref{thm:Mfracorlicz}
and~\ref{thm:Maltwoweight}}

\begin{proof}[Proof of Theorem \ref{thm:Mfracorlicz}]
  Our proof is adapted from the argument for the case
$\alpha=0$ given in~\cite{perez95}.  Fix a non-negative
  function $f$; by a standard approximation argument we may assume
  that the support of $f$ is contained in some cube $Q$.  Let
  $\Phi_p(t)=\Phi(t^{1/p})$;  then be a rescaling argument
  (see~\cite[p.~98]{MR2797562}),
  $\|f^p\|_{Q,\Phi_p}=\|f\|_{Q,\Phi}^p$, and so
\[ M_{\alpha,\Phi}f(x)^q = M_{p\alpha,\Phi_p}(f^p)(x)^{\frac{q}{p}}. \]

By~\cite[Lemma~5.49]{MR2797562}, we have that 
\[ |\{ x\in Q : M_{p\alpha,\Phi_p}(f^p)(x)>t \}|^{\frac{n-p\alpha}{n}}
\leq C\int_{\{ x\in Q : f(x)>t/c\}}
\Phi_p\left(\frac{f(x)^p}{t}\right)\,dx. \]
By the Sobolev relationship, $\frac{n-p\alpha}{n}=\frac{p}{q}$.
Therefore, we have that
\begin{align*}
\left(\int_Q M_{\alpha,\Phi}f(y)^q\,dy\right)^{\frac{1}{q}}
& = \left(\int_Q
  M_{p\alpha,\Phi_p}(f^p)(y)^{\frac{q}{p}}\,dy\right)^{\frac{1}{q}} \\
& \simeq \left(\int_0^\infty t^{\frac{q}{p}}
|\{ x\in Q : M_{p\alpha,\Phi_p}(f^p)(x)>t \}|\frac{dt}{t}\right)
^{\frac{1}{q}} \\
& \lesssim \left(\int_0^\infty t^{\frac{q}{p}}
\bigg( \int_{\{ x\in Q : f(x)^p>t/c\}}
\Phi_p\left(\frac{f(x)^p}{t}\right)\,dx\bigg)^{\frac{q}{p}}\frac{dt}{t}\right)
^{\frac{1}{q}}; \\ 
\intertext{by Minkowski's inequality and the change of variables
  $t\mapsto (f(x)/s)^p$,}
& \lesssim \left(\int_Q \bigg(\int_0^{cf(x)^p} 
\Phi\left(\frac{f(x)}{t^{\frac{1}{p}}} \right)^{\frac{q}{p}}
t^{\frac{q}{p}}\frac{dt}{t}\bigg)^{\frac{p}{q}}\,dx\right)
^{\frac{1}{p}} \\
& = \left(\int_Q \bigg(\int_c^\infty \Phi(s)^{\frac{q}{p}}
\bigg(\frac{f(x)}{s}\bigg)^q\frac{ds}{s}\bigg)^{\frac{p}{q}}\,dx\right)
^{\frac{1}{p}} \\
& = \left(\int_c^\infty
  \frac{\Phi(s)^{\frac{q}{p}}}{s^q}\frac{ds}{s}\right)^{\frac{1}{q}}
\left(\int_Q f(x)^p\,dx\right) ^{\frac{1}{p}}. 
\end{align*}
\end{proof}

\begin{proof}[Proof of Theorem \ref{thm:Maltwoweight}] 
Arguing as
we did for inequality \eqref{dyadicboundM} it
suffices to work with dyadic operators.
Fix a sparse family $\Sp$;  we will estimate 
$$\int_{\R^n}L^\Sp_\al (f\sigma)^qu\,dx
=\sum_{Q\in \Sp}\Big(|Q|^{\frac{\al}{n}}\dashint_Q f\sigma\Big)^q u(E_Q).$$
Let $\beta=n(\frac1p-\frac1q)$.  Then
\begin{align*}
\sum_{Q\in \Sp}\Big(|Q|^{\frac{\al}{n}}\dashint_Q f\sigma\Big)^q u(E_Q)&\leq \sum_{Q\in \Sp} \left(|Q|^{\frac{\al}{n}+\frac1q}\Big(\dashint_Q u\,dx\Big)^{1/q} \dashint_Qf\sigma\,dx\right)^q\\
 &\leq\sum_{Q\in \Sp} \left(|Q|^{\frac{\al}{n}+\frac1q}\Big(\dashint_Q u\,dx\Big)^{1/q}\|\sigma^{\frac{1}{p'}}\|_{\bar{\Phi},Q} \|f\sigma^{\frac1p}\|_{\bar{\Phi},Q} \right)^q\\
 &\leq [u,\sigma]_{A_{p,q,\Phi}^\al}^q\sum_{Q\in \Sp} \|f\sigma^{\frac1p}\|_{\bar{\Phi},Q}^q |Q|^{\frac{q}{p}}\\
 &\lesssim [u,\sigma]_{A_{p,q,\Phi}^\al}^q\sum_{Q\in \Sp} (|Q|^{\frac{\beta}{n}} \|f\sigma^{\frac1p}\|_{\bar{\Phi},Q})^q |E_Q|\\
 &\leq [u,\sigma]_{A_{p,q,\Phi}^\al}^q\int_{\R^n} (M^\D_{\beta,\bar\Phi}f)^q\,dx\\
 &\leq  [u,\sigma]_{A_{p,q,\Phi}^\al}^q\|M_{\beta,\bar{\Phi}}\|_{L^p\ra L^q}^q\|f\|_{L^p(\sigma)}^q.
  \end{align*}
\end{proof}

\bibliographystyle{plain}
\bibliography{frac-MW}

\end{document}